\documentclass[12pt,reqno]{amsart}
\usepackage{latexsym}
\usepackage{amssymb}
\usepackage{mathrsfs}
\usepackage{amsmath}
\usepackage{comment}
\usepackage{fancybox,color}
\usepackage{enumerate}
\usepackage[utf8]{inputenc}
\usepackage{soul}
\usepackage[colorlinks=true, linkcolor=blue, citecolor=blue]{hyperref}

{\catcode`p =12 \catcode`t =12 \gdef\eeaa#1pt{#1}} \def\accentadjtext#1{\setbox0\hbox{$#1$}\kern \expandafter\eeaa\the\fontdimen1\textfont1 \ht0}
\def\accentadjscript#1{\setbox0\hbox{$#1$}\kern \expandafter\eeaa\the\fontdimen1\scriptfont1 \ht0}
\def\accentadjscriptscript#1{\setbox0\hbox{$#1$}\kern \expandafter\eeaa\the\fontdimen1\scriptscriptfont1 \ht0}
\def\accentadjtextback#1{\setbox0\hbox{$#1$}\kern -\expandafter\eeaa\the\fontdimen1\textfont1 \ht0}
\def\accentadjscriptback#1{\setbox0\hbox{$#1$}\kern -\expandafter\eeaa\the\fontdimen1\scriptfont1 \ht0}
\def\accentadjscriptscriptback#1{\setbox0\hbox{$#1$}\kern -\expandafter\eeaa\the\fontdimen1\scriptscriptfont1 \ht0}
\def\itoverline#1{{\mathsurround0pt\mathchoice
        {\rlap{$\accentadjtext{\displaystyle #1}
                \accentadjtext{\vrule height1.593pt}
                \overline{\phantom{\displaystyle #1}
                \accentadjtextback{\displaystyle #1}}$}{#1}}
        {\rlap{$\accentadjtext{\textstyle #1}
                \accentadjtext{\vrule height1.593pt}
                \overline{\phantom{\textstyle #1}
                \accentadjtextback{\textstyle #1}}$}{#1}}
        {\rlap{$\accentadjscript{\scriptstyle #1}
                \accentadjscript{\vrule height1.593pt}
                \overline{\phantom{\scriptstyle #1}
                \accentadjscriptback{\scriptstyle #1}}$}{#1}}
        {\rlap{$\accentadjscriptscript{\scriptscriptstyle #1}
                \accentadjscriptscript{\vrule height1.593pt}
                \overline{\phantom{\scriptscriptstyle #1}
                \accentadjscriptscriptback{\scriptscriptstyle #1}}$}{#1}}}}

\newcommand{\iol}{\itoverline}

\newcommand{\ch}[1]{{\mbox{\raise 1pt\hbox{\large$\chi$}}}_{\lower 1pt\hbox{$\scriptstyle #1$}}}

\usepackage[left=2.8cm,top=2.5cm,right=2.8cm,bottom=2.5cm]{geometry}

\def\1{\raisebox{2pt}{\rm{$\chi$}}}

\newtheorem{theorem}{Theorem}[section]

\newtheorem{lemma}[theorem]{Lemma}

\theoremstyle{definition}
\newtheorem{definition}[theorem]{Definition}
\newtheorem{remark}[theorem]{Remark}

\IfFileExists{mathx10.tfm}{\DeclareFontFamily{U}{mathx}{}
  \DeclareFontShape{U}{mathx}{m}{n}{<-> mathx10}{}
  \DeclareSymbolFont{mathx}{U}{mathx}{m}{n}
  \DeclareMathAccent{\widehat}{0}{mathx}{"70}
  \DeclareMathAccent{\widecheck}{0}{mathx}{"71}
}{}

\newcommand{\R}{{\mathbb R}}

\newcommand{\Z}{{\mathbb Z}}

\newcommand{\D}{{\mathcal D}}

\def\1{\raisebox{2pt}{\rm{$\chi$}}}

\def\cprime{$'$}

\def\vint_#1{\mathchoice {\mathop{\kern 0.2em\vrule width 0.6em height 0.69678ex depth -0.58065ex
                \kern -0.8em \intop}\nolimits_{\kern -0.4em#1}}{\mathop{\kern 0.1em\vrule width 0.5em height 0.69678ex depth -0.60387ex
                \kern -0.6em \intop}\nolimits_{#1}}{\mathop{\kern 0.1em\vrule width 0.5em height 0.69678ex
            depth -0.60387ex
                \kern -0.6em \intop}\nolimits_{#1}}{\mathop{\kern 0.1em\vrule width 0.5em height 0.69678ex depth -0.60387ex
                \kern -0.6em \intop}\nolimits_{#1}}}
\def\vintslides_#1{\mathchoice {\mathop{\kern 0.1em\vrule width 0.5em height 0.697ex depth -0.581ex
                \kern -0.6em \intop}\nolimits_{\kern -0.4em#1}}{\mathop{\kern 0.1em\vrule width 0.3em height 0.697ex depth -0.604ex
                \kern -0.4em \intop}\nolimits_{#1}}{\mathop{\kern 0.1em\vrule width 0.3em height 0.697ex depth -0.604ex
                \kern -0.4em \intop}\nolimits_{#1}}{\mathop{\kern 0.1em\vrule width 0.3em height 0.697ex depth -0.604ex
                \kern -0.4em \intop}\nolimits_{#1}}}

\newcommand{\kint}{\vint}

\newcommand{\ud}{\, d}

\newcommand{\dist}{\operatorname{dist}}

\title[Median porosity is QC invariant]{Median porosity is quasiconformally invariant}
\thanks{The proof development and the preparation of this note were assisted by the large language models ChatGPT and Claude.}
\author[T. Kilpel\"ainen]{Tero Kilpel\"ainen}
\address[T.K.]{University of Jyvaskyla, Department of Mathematics and Statistics, P.O. Box 35, FI-40014 University of Jyvaskyla, Finland}
\email{tero.kilpelainen@jyu.fi}

\author[A.V.\! V\"ah\"akangas]{Antti V. V\"ah\"akangas}
\address[A.V.V.]{University of Jyvaskyla, Department of Mathematics and Statistics, P.O. Box 35, FI-40014 University of Jyvaskyla, Finland}
\email{antti.vahakangas@iki.fi}

\makeatletter
\@namedef{subjclassname@2020}{{\mdseries 2020} Mathematics Subject Classification}
\makeatother



\begin{document}

\begin{abstract}
A set in~$\R^n$ is median porous if the
logarithm of its distance function has bounded mean oscillation.
We show that this property is preserved under quasiconformal
mappings. 
In particular, median porosity is quasiconformally invariant.
We also show that the stronger notion of weak porosity, by contrast, is not quasiconformally invariant.
\end{abstract}
\maketitle

\section{Introduction}

A set $E\subset\R^n$ is {\it porous} if every ball contains a smaller ball, with comparable radius, that does not touch $E$.
Recently, two weaker notions of porosity have been
characterized in terms of the distance function
$\delta_E=\dist(\cdot,E)$: a set~$E$ is {\it weakly porous} (Definition~\ref{def:weakly-porous})
if and only if $\delta_E^{-\alpha}\in A_1$ for some $\alpha>0$~\cite{Vasin2003,MR4773553},
and {\it median porous} (in the sense of~\cite{PU25}) if and only if
$\delta_E^{-\alpha}\in A_\infty$ for some $\alpha>0$.
These notions satisfy
\[
E\text{ is porous}\implies E\text{ is weakly porous}\implies E\text{ is median porous}\,,
\]
and both implications are strict~\cite{MR4773553,PU25}.
See also~\cite{MR4966412} for related characterizations of $A_p$ distance weights.

It follows from standard distortion theorems that entire quasiconformal mappings
preserve porosity; V\"ais\"al\"a~\cite{MR869219} extended this to quasisymmetric mappings
defined only on the porous set itself. In this note
we prove that median porosity is invariant under entire quasiconformal mappings.
We also show that weak porosity is not quasiconformally invariant (see Theorem~\ref{thm:sharpness}).

\begin{theorem}\label{thm:main}
Let $n\ge 2$ and let $f\colon\R^n\to\R^n$ be a quasiconformal mapping.
If $E\subset\R^n$ is median porous, then $f(E)$ is median porous.
\end{theorem}

The median porosity
of $E$ is equivalent to
$\log\delta_E\in\mathrm{BMO}(\R^n)$ by the standard characterization of $A_\infty$; see \cite{MR807149}
and \cite[Theorem 1.5]{PU25}.
Thus Theorem~\ref{thm:main} reduces to showing that this property is
preserved under quasiconformal mappings; we will establish this in
Theorem~\ref{thm:QC_invariance}, which is a quantitative variant of Theorem~\ref{thm:main}.
We will employ Reimann's $\mathrm{BMO}$ isomorphism theorem~\cite{MR361067}, which yields
$\log\delta_E\circ f^{-1}\in\mathrm{BMO}$. In general this function does not
coincide with $\log\delta_{f(E)}$. However, the resulting defect
\[
\log\delta_{f(E)}-\log\delta_E\circ f^{-1}
\] is shown to lie in
$\mathrm{BMO}$ by using Astala--Gehring's quasiconformal Koebe distortion
theorem~\cite{MR777305} together with a variable-radius averaging theorem
for BMO. The last result is established in Section~2
(Theorem~\ref{thm:main_averaging}) and applied in Section~3 to complete
the proof of Theorem~\ref{thm:main}.

\section{BMO averaging theorem}

We start {by} proving a variable-radius averaging theorem for BMO (Theorem~\ref{thm:main_averaging}), applied in Section~3 to establish Theorem~\ref{thm:main}.

For $u\in L^1_{\mathrm{loc}}(\R^n)$, we denote its {\it $\mathrm{BMO}$ seminorm} by
\[
\|u\|_{\mathrm{BMO}}=\sup_B\,\kint_B |u(x)-u_B|\ud x\,,
\]
where $u_B=\kint_B u(y)\ud y$ and the supremum is taken over all balls $B\subset\R^n$.
We denote $u\in \mathrm{BMO}=\mathrm{BMO}(\R^n)$ if $\|u\|_{\mathrm{BMO}}<\infty$
and $u\in L^1_{\mathrm{loc}}(\R^n)$.
We refer to~\cite{MR807149,MR1232192} for background on BMO and related topics.

\begin{theorem}[John--Nirenberg inequality {\cite{MR131498}}]\label{thm:JN}
Let $u\in\mathrm{BMO}$. Then for all balls $B\subset\R^n$ and all $\lambda\ge 0$,
\[
|\{x\in B:|u(x)-u_B|>\lambda\}|
\le c_1e^{-(c_2/\|u\|_{\mathrm{BMO}})\,\lambda} \cdot |B|\,,
\]
where the constants $c_1,c_2>0$ depend only on the dimension $n$.
\end{theorem}

 The following two elementary lemmas are known. Nevertheless, for the reader's convenience, we include short proofs. See also~\cite{MR807149,MR1232192}.

\begin{lemma}\label{lem:P}
Let $u\in\mathrm{BMO}$ and $\rho>0$.
Define
\[
A_\rho u(x)=\kint_{B(x,\rho)} u(y)\ud y\,,\quad x\in\R^n\,.
\]
Then $A_\rho u\in\mathrm{BMO}$ and $
\|A_\rho u\|_{\mathrm{BMO}}\le \|u\|_{\mathrm{BMO}}$.
\end{lemma}

\begin{proof}
By continuity, we have $A_\rho u\in L^1_{\mathrm{loc}}(\R^n)$.
The change of variable $y=x+w$ gives
\[
A_\rho u(x)=\kint_{B(0,\rho)} u(x+w)\ud w\,.
\]
Fix a ball $B\subset\R^n$. By Fubini's theorem and a further change of variable $z=x+w$,
\[
(A_\rho u)_B
=\kint_{B(0,\rho)}\kint_B u(x+w)\ud x\ud w
=\kint_{B(0,\rho)} u_{B+w}\ud w\,.
\]
The triangle inequality and Fubini's theorem therefore yield
\begin{equation}\label{eq:est1}
\kint_B|A_\rho u(x)-(A_\rho u)_B|\ud x
\le
\kint_{B(0,\rho)}\kint_B|u(x+w)-u_{B+w}|\ud x\ud w\,.
\end{equation}
The same change of variable $z=x+w$ identifies the inner integral as the
oscillation of $u$ over the translated ball $B+w$:
\[
\kint_B|u(x+w)-u_{B+w}|\ud x
=\kint_{B+w}|u(z)-u_{B+w}|\ud z
\le\|u\|_{\mathrm{BMO}}\,.
\]
Combining this with \eqref{eq:est1}, we arrive at
\[
\kint_B|A_\rho u(x)-(A_\rho u)_B|\ud x\le\|u\|_{\mathrm{BMO}}\,,
\]
as desired.
\end{proof}

\begin{lemma}\label{lem:C}
Let $u\in\mathrm{BMO}$ and $0<r,R<\infty$. Then
\begin{equation}\label{e.tavoite}
|A_r u(x)-A_R u(x)|\le C\Bigl(1+\Bigl|\log\frac{R}{r}\Bigr|\Bigr)\|u\|_{\mathrm{BMO}}\,,
\end{equation}
for all $x\in\R^n$; here constant $C$ depends only on the dimension $n$.
\end{lemma}

\begin{proof}
It is straightforward to show that
\begin{equation}\label{eq:local-doubling}
|A_{\rho_1}u(x)-A_{\rho_2}u(x)|\le \left(\frac{\rho_2}{\rho_1}\right)^n\|u\|_{\mathrm{BMO}}\,,
\end{equation}
for every $x\in\R^n$ and every pair of radii $0<\rho_1\le \rho_2$.

By symmetry of both sides of \eqref{e.tavoite} in $r$ and $R$, we may assume $0<r\le R$.
Let $N\ge 1$ be the unique integer satisfying
\[
e^{N-1}r\le R<e^N r\,.
\]
Then $N\le 1 +\log(R/r)$.

Letting $r_k=e^k r$ for $k=0,\ldots,N-1$ and $r_N=R$, these radii obey
\[1\le r_k/r_{k-1}\le e\,,\quad k=1,\ldots,N.\]
Thus using telescoping summation and~\eqref{eq:local-doubling} for consecutive radii $r_k$, we obtain
\begin{align*}
|A_r u(x)-A_R u(x)|
&\le\sum_{k=1}^{N}|A_{r_{k-1}}u(x)-A_{r_k}u(x)|
\\&\le N\cdot e^n\|u\|_{\mathrm{BMO}}
\le e^n\Bigl(1+\log\frac{R}{r}\Bigr)\|u\|_{\mathrm{BMO}}\,.
\end{align*}
This proves the lemma with $C=e^n$.
\end{proof}

Throughout the note, we abbreviate $\delta_E=\dist(\cdot ,E)$ whenever $ \emptyset\ne E\subset \R^n$.
Observe that $\log\delta_E\in\mathrm{BMO}$ implies $\log\delta_E\in L^1_{\mathrm{loc}}(\R^n)$.
 Hence $\delta_E>0$ almost everywhere; in particular $\lvert \iol{E}\rvert=0$.

Lemma~\ref{lem:P} shows that fixed-radius averaging preserves BMO.
The following theorem extends this to variable-radius averaging,
where the radius at each point is given by the distance to a set.
Variable-step mollifiers and variable-radius averaging operators of this
kind have a long history in the Sobolev and trace-theory literature; see
Burenkov~\cite{MR1622690} and Shan\cprime kov~\cite{MR816512}.
The following result might be known, but we were not able to locate it in the literature.

\begin{theorem}\label{thm:main_averaging}
Assume that $\emptyset\ne E\subset\R^n$ is such that
$\log\delta_E\in\mathrm{BMO}$.
Let $v\in\mathrm{BMO}$ and define
\[
A_E v(x)=\kint_{B(x,\,\delta_E(x))} v(y)\ud y\,,\quad x\in\R^n\setminus \iol{E}\,.
\]
Then $A_E v\in\mathrm{BMO}$ and
\[
\|A_E v\|_{\mathrm{BMO}}
\le C(n,\|\log\delta_E\|_{\mathrm{BMO}})\,\|v\|_{\mathrm{BMO}}\,.
\]
\end{theorem}

\begin{proof}
Clearly, we may assume that $E$ is closed.
Then
$\lVert \log \delta_E\rVert_{\mathrm{BMO}}>0$. 
Denote $u=\log \delta_E$ and
fix a ball $B\subset\R^n$.
We write
\[
a=u_B=\kint_B u(y)\ud y\in\R\quad \text{ and }\quad \rho=e^a>0.
\]
For $x\in B\setminus E$, we have
\[
u(x)-a=\log\frac{\delta_E(x)}{\rho}\,.
\]
Define the layers
\[
L_k=\{x\in B\setminus E: e^{k}\rho<\delta_E(x)\le e^{k+1}\rho\}\,,\quad k\in\Z\,.
\]
Since $E$ is closed, these layers form a disjoint cover of
$B\setminus E$.

Next we show that
\begin{equation}\label{eq:Lk_measure}
|L_k|\le C\,e^{-c\,|k|}\,|B|\qquad\text{for all }k\in\Z\,,
\end{equation}
where $C=C(n,\|u\|_{\mathrm{BMO}})$ and $c=c(n,\|u\|_{\mathrm{BMO}})>0$.

Let $c_1$ and $c_2$ be the constants from the John--Nirenberg inequality
(Theorem~\ref{thm:JN}), and set $c_3=c_2/(2\lVert u\rVert_{\mathrm{BMO}})>0$.
Fix $\lambda>0$. For every $x\in B\setminus E$,
\begin{align*}
\delta_E(x)\le e^{-\lambda}\rho &\implies u(x)-a\le-\lambda\,,\\
\delta_E(x)> e^{\lambda}\rho &\implies u(x)-a>\lambda\,;
\end{align*}
in either case $|u(x)-a|\ge\lambda$, so $x\in\{y\in B:|u(y)-a|>\lambda/2\}$.
The John--Nirenberg inequality therefore yields
\begin{equation}\label{eq:sublevel}
|\{x\in B\setminus E:\delta_E(x)\le e^{-\lambda}\rho\}|
+|\{x\in B\setminus E:\delta_E(x)> e^{\lambda}\rho\}|
\le c_1\,e^{-c_3\lambda}\,|B|\,,
\end{equation}
for all $\lambda>0$. We deduce~\eqref{eq:Lk_measure} from~\eqref{eq:sublevel} in three cases:
\begin{itemize}
\item For $|k|\le 1$ the bound~\eqref{eq:Lk_measure} is immediate:
$|L_k|\le|B|\le e^{c_3}\,e^{-c_3|k|}\,|B|$.
\item For $k\ge 2$, the inclusion
$L_k\subset\{x\in B\setminus E:\delta_E(x)>e^{k}\rho\}$
and~\eqref{eq:sublevel} at $\lambda=k$ give
$|L_k|\le c_1\,e^{-c_3 \lvert k\rvert}\,|B|$.
\item For $k\le-2$, the inclusion
$L_k\subset\{x\in B\setminus E:\delta_E(x)\le e^{-(|k|-1)}\rho\}$
and~\eqref{eq:sublevel} at $\lambda=|k|-1$ give
$|L_k|\le c_1\,e^{c_3}\,e^{-c_3|k|}\,|B|$.
\end{itemize}
Taking $C=e^{c_3}\max\{1,c_1\}$ and $c=c_3$ proves~\eqref{eq:Lk_measure}.

If $x\in L_k$, then
\[
k<\log\frac{\delta_E(x)}{\rho}\le k+1,
\]
whence
\[
\left|\log\frac{\delta_E(x)}{\rho}\right|\le |k|+1.
\]
Lemma~\ref{lem:C}, applied with radii $\delta_E(x)$ and $\rho$, gives
\begin{equation}\label{eq:Lk_pointwise}
|A_E v(x)-A_\rho v(x)|
\le C(n)\Bigl(1+\Bigl|\log\frac{\delta_E(x)}{\rho}\Bigr|\Bigr)\|v\|_{\mathrm{BMO}}
\le C(n)(2+|k|)\,\|v\|_{\mathrm{BMO}}
\end{equation}
for all $x\in L_k$.

Combining~\eqref{eq:Lk_measure} and~\eqref{eq:Lk_pointwise}, and using the fact that $E$ has measure zero, we have
\begin{align*}
\kint_B|A_E v(x)-A_\rho v(x)|\ud x
&=\frac{1}{|B|}\sum_{k\in\Z}\int_{L_k}|A_E v(x)-A_\rho v(x)|\ud x\\
&\le C(n,\|u\|_{\mathrm{BMO}})\,\|v\|_{\mathrm{BMO}}
\sum_{k\in\Z}(2+|k|)\,e^{-c\,|k|}\\
&\le C(n,\|u\|_{\mathrm{BMO}})\,\|v\|_{\mathrm{BMO}}\,.
\end{align*}

By the above estimate and Lemma \ref{lem:P},
\begin{align*}
&\kint_B|A_E v(x)-(A_E v)_B|\ud x
\le 2\kint_B|A_E v(x)-(A_\rho v)_B|\ud x\\
&\qquad\le 2\kint_B|A_E v(x)-A_\rho v(x)|\ud x
+2\kint_B|A_\rho v(x)-(A_\rho v)_B|\ud x\\
&\qquad\le C(n,\|u\|_{\mathrm{BMO}})\,\|v\|_{\mathrm{BMO}}\,.
\end{align*}
The claim follows by taking supremum over all balls $B$ in $\R^n$.
\end{proof}

\section{Invariance under quasiconformal mappings}

A homeomorphism $f\colon D\to D'$ between domains $D$ and $D'$ of $\R^n$
is a {\it $K$-qua\-si\-con\-for\-mal mapping} if $f\in W^{1,n}_{\mathrm{loc}}(D,\R^n)$ and
\[|Df(x)|^n\le K\,J_f(x)\quad \text{ for almost every } x\in D\,; \]
we refer to~\cite{MR454009,MR1859913} for background.
In particular $J_f\ge 0$ almost everywhere.

We shall use two results from the theory of quasiconformal mappings. The first is Reimann's $\mathrm{BMO}$ isomorphism theorem.

\begin{theorem}[Reimann {\cite{MR361067}}]\label{thm:Reimann}
Let $f\colon\R^n\to\R^n$ be a $K$-quasiconformal mapping. Then:
\begin{enumerate}[\upshape(a)]
\item
The composition operator $u\mapsto u\circ f^{-1}$ is a bounded
isomorphism of\/ $\mathrm{BMO}$:
\[
\|u\circ f^{-1}\|_{\mathrm{BMO}}
\le C\,\|u\|_{\mathrm{BMO}}
\qquad\text{for all }u\in\mathrm{BMO}\,,
\]
where $C=C(K,n)$.
\item
$\log J_f\in\mathrm{BMO}$ with
$\|\log J_f\|_{\mathrm{BMO}}\le C(K,n)$.
\end{enumerate}
\end{theorem}

We also employ the following Astala--Gehring quasiconformal Koebe distortion
theorem~\cite{MR777305}.

\begin{theorem}[Astala--Gehring {\cite{MR777305}}]\label{thm:AG}
Let $D$ and $D'$ be proper subdomains of $\R^n$ and let $f\colon D\to D'$ be
a $K$-quasiconformal mapping. Then
\[
\frac{1}{c}\,\frac{\dist(f(x),\partial D')}{\dist(x,\partial D)}
\le a_f(x)
\le c\,\frac{\dist(f(x),\partial D')}{\dist(x,\partial D)}
\]
for all $x\in D$, where $c=c(K,n)$ and
\[
a_f(x)=\exp\left(\frac{1}{n}\kint_{B(x,\,\dist(x,\partial D))}
\log J_f(y)\ud y\right)\,.
\]
\end{theorem}

The following theorem is a quantitative variant of Theorem~\ref{thm:main}.

\begin{theorem}\label{thm:QC_invariance}
Suppose that  $ E\subset\R^n$, $n\ge 2$, is a nonempty set with $\log\delta_E\in\mathrm{BMO}$.
If $f\colon\R^n\to\R^n$ is a $K$-quasiconformal mapping,
then $\log\delta_{f(E)}\in\mathrm{BMO}$ and  
\[
\|\log\delta_{f(E)}\|_{\mathrm{BMO}}
\le C(K,n,\|\log\delta_E\|_{\mathrm{BMO}})\,.
\]
\end{theorem}

\begin{proof}
Clearly, we may assume that $E$ is closed.
Denote $F=f(E)$ and note that $F$ is closed. Since $\log\delta_E\in\mathrm{BMO}$, we have $\lvert E\rvert=0$, and the Lusin condition $(N)$ then gives $\lvert F\rvert=0$
(see e.g.~\cite[Theorem~33.2]{MR454009}).
Denote
\[
u_E=\log\delta_E\quad\text{in }\R^n\setminus E\,,\qquad
u_F=\log\delta_F\quad\text{in }\R^n\setminus F\,.
\]
Then $u_E$ and $u_F$ are both defined almost everywhere in~$\R^n$.

Consider the decomposition
\[
u_F=(u_F-u_E\circ f^{-1})+u_E\circ f^{-1}\,.
\]
Theorem~\ref{thm:Reimann}(a) implies that $u_E\circ f^{-1}\in\mathrm{BMO}$.
It remains to show that the defect $u_F-u_E\circ f^{-1}$ belongs to $\mathrm{BMO}$.

For $x\in\R^n\setminus E$, let $G$ stand for  
the connected component of
$\R^n\setminus E$ containing~$x$. Then the restriction $f|_G\colon G\to f(G)$ is
$K$-quasiconformal, and
\begin{align*}
\delta_E(x)&=\dist(x,\R^n\setminus G)=\dist(x,\partial G)\,,\\
\delta_F(f(x))&=\dist(f(x),\R^n\setminus f(G))=\dist(f(x),\partial f(G))\,.
\end{align*}
Theorem~\ref{thm:AG} applied to $D=G$ and $D'=f(G)$ gives a constant $c=c(K,n)\ge 1$ such that
\begin{equation}\label{e.A-G}
\frac{1}{c}\,\frac{\delta_F(f(x))}{\delta_E(x)}
\le a_f(x)
\le c\cdot\frac{\delta_F(f(x))}{\delta_E(x)}
\end{equation}
for all $x\in\R^n\setminus E$, where
\[
a_f(x)=\exp\left(\frac{1}{n}\kint_{B(x,\delta_E(x))}\log J_f(y)\ud y\right)\,.
\]
Denote 
\[
A_E(\log J_f)(x)=\kint_{B(x,\delta_E(x))}\log J_f(y)\ud y
\]
and
\[
b(x)=\log\frac{\delta_F(f(x))}{\delta_E(x)}-\frac{1}{n}A_E(\log J_f)(x)
\]
for $x\in\R^n\setminus E$.
Since $\lvert E\rvert=0$, we have by
 taking logarithms in the Astala--Gehring estimate \eqref{e.A-G} and rearranging that
\[
\|b\|_{L^\infty(\R^n)}\le\log c\,.
\] 
Hence $b\in L^\infty(\R^n)\subset\mathrm{BMO}$ with 

\[\|b\|_{\mathrm{BMO}}\le 2\|b\|_{L^\infty(\R^n)}\le 2\log c\,.\]

We now introduce the defect. Define $\Phi\colon\R^n\to\R$ by
\[
\Phi(z)=u_F(z)-(u_E\circ f^{-1})(z)
\]
for $z\in\R^n\setminus F$, and extend $\Phi$ by zero in $F$.
For $x\in\R^n\setminus E$ we have that
\begin{align*}
(\Phi\circ f)(x)
&=u_F(f(x))-u_E(f^{-1}(f(x)))\\
&=u_F(f(x))-u_E(x)\\
&=\log\delta_F(f(x))-\log\delta_E(x)\\
&=\log\frac{\delta_F(f(x))}{\delta_E(x)}
 \,.
\end{align*}
Thus 
\begin{equation}\label{eq:key_identity}
\Phi\circ f=\frac{1}{n}A_E(\log J_f)+b\quad\text{in }\R^n\setminus E\,.
\end{equation}

Now Reimann's BMO isomorphism Theorem~\ref{thm:Reimann}(b) ensures that  $\log J_f\in\mathrm{BMO}$.
Consequently, $A_E(\log J_f)\in\mathrm{BMO}$ by Theorem~\ref{thm:main_averaging}
since also $\log\delta_E\in\mathrm{BMO}$.
This combined with $b\in \mathrm{BMO}$ and~\eqref{eq:key_identity},
 guarantees that $\Phi\circ f\in\mathrm{BMO}$.
Now, employing again  Reimann's isomorphism Theorem~\ref{thm:Reimann}(a) we   have
\[
\Phi=(\Phi\circ f)\circ f^{-1}\in\mathrm{BMO}\,
\]
and, finally, arrive at the desired conclusion  $u_F=\Phi+u_E\circ f^{-1}\in\mathrm{BMO}$. 
\end{proof}

By the equivalence between $\log\delta_E\in\mathrm{BMO}(\R^n)$ and median porosity recalled in the introduction, Theorem~\ref{thm:QC_invariance} proves Theorem~\ref{thm:main}.

\begin{remark}
It seems plausible that Theorem~\ref{thm:QC_invariance} extends to
$K$-quasiconformal mappings of the first Heisenberg group $\mathbb H^1$,
with the required Reimann and Astala--Gehring estimates supplied
by~\cite{MR3453992,MR4065111}. We do not pursue this here.
\end{remark}

\begin{remark}
The proof of Theorem~\ref{thm:QC_invariance} can also be cast at the
$A_\infty$ level using Uchiyama's invariance theorem~\cite{MR442226}; the
Jacobian factor in the  push forward $w\mapsto(w\circ f^{-1})\,J_{f^{-1}}$
then has to be controlled separately, and the resulting argument is more
involved than the BMO version above.
\end{remark}

\section{Weak porosity is not quasiconformally invariant}

Recall the porosity hierarchy for  a set $E\subset\R^n$
\[
E\text{ is porous}\implies E\text{ is weakly porous}\implies E\text{ is median porous} 
\]
and  the following Muckenhoupt characterizations: $E$ is weakly porous (Definition~\ref{def:weakly-porous}) if and only if $\delta_E^{-\alpha}\in A_1$ for some $\alpha>0$~\cite{Vasin2003,MR4773553}; $E$ is median porous if and only if $\delta_E^{-\alpha}\in A_\infty$ for some $\alpha>0$~\cite{PU25}. The latter is preserved under quasiconformal maps of $\R^n$, $n\ge 2$ 
(Theorem~\ref{thm:main}).
We 
 show the former is not: for every $K>1$, some $K$-quasiconformal map of $\R^n$ fails to preserve weak porosity.
More precisely, we establish: 

\begin{theorem}\label{thm:sharpness}
Let $n\ge 2$ and $K>1$. There exist a weakly porous set $E\subset\R^n$ and a $K$-quasiconformal mapping $f\colon\R^n\to\R^n$ such that the image $f(E)$ is not weakly porous but  is median porous. 
\end{theorem}

For the proof we first recall the relevant definitions and structural results from the theory of weakly porous sets.

Following~\cite[\S2]{MR4773553}, we work with
{\it half-open} cubes whose sides are parallel to the coordinate axes,
that is, sets of the form
\[
Q=[a_1,b_1)\times\cdots\times[a_n,b_n),\qquad
\ell(Q)=b_1-a_1=\cdots=b_n-a_n.
\]
The {\it dyadic decomposition} of a cube $P\subset\R^n$ is
\[
\D(P)={\textstyle \bigcup\limits_{j=0}^{\infty}}\D_j(P),
\]
where $\D_j(P)$ consists of the $2^{jn}$ pairwise disjoint half-open
cubes of side length $2^{-j}\ell(P)$ that partition $P$; in particular,
$\D_0(P)=\{P\}$. The cubes in $\D(P)$ are nested in the sense that
$P_1\cap P_2\in\{P_1,P_2,\emptyset\}$ for all $P_1,P_2\in\D(P)$, and every
$Q\in\D_j(P)$ with $j\ge 1$ is contained in a unique cube of
$\D_{j-1}(P)$, which we call its {\it dyadic parent}.

\begin{definition}[Definition 3.1 of~\cite{MR4773553}]\label{def:weakly-porous}
Let $E\subset\R^n$ be a nonempty set. 
\begin{itemize}
\item A subcube $Q\in\D(P)$ of a cube $P\subset\R^n$ is said to be {\it $E$-free} if $E\cap Q=\emptyset$, and we write $\mathcal M(P)\in\D(P)$ for an $E$-free dyadic subcube of $P$ with largest side length (fixed if not unique).
\item The set $E$ is called {\it weakly porous} if there exist constants $0<c,\delta<1$ such that for every cube $P\subset\R^n$ there exist finitely many pairwise disjoint $E$-free subcubes $Q_1,\ldots,Q_N\in\D(P)$ with
\[
|Q_k|\ge \delta|\mathcal M(P)| \quad\text{for all $k=1,\ldots,N$,}
\qquad\text{and}\qquad
\sum_{k=1}^N |Q_k|\ge c|P|\,.
\]
\end{itemize}
\end{definition}

We now start the construction for Theorem~\ref{thm:sharpness} with
 the set
\[
E={\textstyle \bigcup\limits_{m=1}^{\infty}}\partial B(0,m).
\]
The uniform gaps in $E$ make it weakly porous.

\begin{lemma}\label{lem:E-weakly-porous}
The set
\[
E={\textstyle \bigcup\limits_{m=1}^{\infty}}\partial B(0,m)
\]
is weakly porous in $\R^n$ for $n\ge 2$.
\end{lemma}

\begin{proof}
Let $P\subset\R^n$ be a cube with $P\cap E\ne\emptyset$. Note that the case
$P\cap E=\emptyset$ is trivial. Write $\ell_0=\ell(\mathcal M(P))$.
We first show that $\ell_0\le 2$. The radial range $\{|x|:x\in\mathcal M(P)\}$
is an interval of length at least $\ell_0/2$ avoiding the positive integers
(since $\mathcal M(P)$ is $E$-free), hence lies in a unit-length
component of $[0,\infty)\setminus\{1,2,3,\ldots\}$; thus $\ell_0\le 2$.

We first establish the following estimate: for every $0<\lambda<1/2$,
\begin{equation}\label{eq:tube-at-M-scale}
\bigl|\{x\in P:\dist(x,E)<\lambda \ell_0\}\bigr|
   \le C_n\lambda |P|.
\end{equation}
Set $r=\lambda \ell_0$, so $r<1$. The standard tubular volume estimate for a
sphere gives, for any integer $m\ge 1$,
\begin{equation}\label{eq:tubular-estimate}
\bigl|\{x\in P:\dist(x,\partial B(0,m))<r\}\bigr|
   \le C_n\,r\,\ell(P)^{n-1}.
\end{equation}
The radial range of $P$ has length at most $\sqrt n\,\ell(P)$, so $P$
comes within distance $r<1$ of at most $\lceil\sqrt n\,\ell(P)\rceil+2$
such spheres. By \eqref{eq:tubular-estimate} and subadditivity,
\[
\bigl|\{x\in P:\dist(x,E)<r\}\bigr|
   \le \bigl(\lceil\sqrt n\,\ell(P)\rceil+2\bigr)\,C_n\,r\,\ell(P)^{n-1}.
\]
For $\ell(P)\le 1$, the count is at most $\lceil\sqrt n\rceil+2$ and
$r\le\lambda\,\ell(P)$ (from $\ell_0\le\ell(P)$); for $\ell(P)>1$, the count
is at most $4\sqrt n\,\ell(P)$ and $r\le 2\lambda$ (from $\ell_0\le 2$).
Either way the right-hand side is at most $C_n\,\lambda\,|P|$, proving
\eqref{eq:tube-at-M-scale}.

Let $\mathcal W(P)$ denote the family of {\it maximal $E$-free dyadic
subcubes} of $P$ --- that is, those $Q\in\D(P)$ with $E\cap Q=\emptyset$
whose dyadic parent in $\D(P)$ meets $E$.
The cubes in $\mathcal W(P)$ are pairwise disjoint,
and since $E$ is closed, \[\bigcup_{Q\in\mathcal W(P)} Q = P\setminus E\,.\]
Fix $\theta>0$, to be chosen below, and call $Q\in\mathcal W(P)$
{\it good} if
\[
\ell(Q)\ge\theta \ell_0,
\]
and {\it bad} otherwise. A bad cube $Q$ is $E$-free, and therefore $\ell(Q)\le \ell_0<\ell(P)$, so its
dyadic parent in $\mathcal D(P)$ exists; the parent meets $E$ by
maximality. Hence
\[
Q\subset\{x\in P:\dist(x,E)\le 2\sqrt n\,\ell(Q)\}
 \subset \{x\in P:\dist(x,E)<2\sqrt n\,\theta \ell_0\}.
\]
Applying \eqref{eq:tube-at-M-scale} with $\lambda=2\sqrt n\,\theta$
(assuming $\theta<1/(4\sqrt n)$) gives
\[
\sum_{\substack{Q\in\mathcal W(P)\\ Q\text{ bad}}}|Q|
   \le |\{x\in P:\dist(x,E)<2\sqrt n\,\theta \ell_0\}|
   \le 2\sqrt n\,C_n\,\theta\,|P|.
\]
Choose $\theta=\theta(n)>0$ so small that
\[
\theta<1/(4\sqrt n)\quad \text{ and }\quad 2\sqrt n\,C_n\,\theta\le 1/2\,.
\]
Since $|E|=0$, $\sum_{Q\in\mathcal W(P)}|Q|=|P|$, so the good cubes carry
total volume at least $|P|/2$. Each good $Q$ is $E$-free, lies in
$\mathcal D(P)$, and has $|Q|\ge\theta^n|\mathcal M(P)|$. These good cubes are pairwise disjoint and, by the lower bound on $|Q|$, finite in number. This realizes
Definition~\ref{def:weakly-porous} with $c=\tfrac12$ and $\delta=\theta^n$,
both less than $1$.
\end{proof}

We now turn to quasiconformal non-invariance. For this, we define the {\it radial stretching} to be the homeomorphism $f_\gamma\colon\R^n\to\R^n$
with
\[
f_\gamma(x)=|x|^{\gamma-1}x\quad\text{for }x\ne 0,
\]
where $0<\gamma<1$.
The mapping $f_\gamma$ is $K$-quasiconformal with $K=1/\gamma$ in the outer-dilatation convention of Section~3 (singular values $\gamma|x|^{\gamma-1}$ radial, $|x|^{\gamma-1}$ tangential; see~\cite[\S 6.5.1]{MR1859913}).
The set
\[
f_\gamma(E)={\textstyle \bigcup\limits_{m=1}^{\infty}}\partial B(0,m^\gamma)
\]
is the image of the set $E$ under the radial stretching $f_\gamma$.

 The annuli $ B(0, (m+1)^{\gamma})\setminus \iol{B}(0, m^{\gamma}) $ diminish too fast
 for $f_\gamma(E)$ to be weakly porous. To verify this, we will use the following doubling property of $|\mathcal M(\,\cdot\,)|$.

\begin{lemma}[Lemma 3.2(ii) of~\cite{MR4773553}]\label{lem:M-doubling}
 Assume that $E\subset\R^n$ is weakly porous with constants $0<c,\delta<1$. Then there exists $C=C(n,c,\delta)>0$ such that
\[
|\mathcal M(R)|\le C\,|\mathcal M(Q)|
\]
whenever $Q\subset R$ are cubes with $|R|=2^n|Q|$.
\end{lemma}

\begin{lemma}\label{lem:image-not-weakly-porous}
For every $\gamma\in(0,1)$ and $n\ge 2$, the set
\[f_\gamma(E)={\textstyle \bigcup\limits_{m=1}^{\infty}}\partial B(0,m^\gamma)
\] is \emph{not} weakly porous in $\R^n$.
\end{lemma}

\begin{proof}
We will observe that the property in Lemma~\ref{lem:M-doubling} fails for $f_\gamma(E)$, whence it cannot be
 weakly porous.

 Fix an integer $k\ge 1$ and set $L=2^{k\gamma}$. We consider cubes
\[
 R_L=[0,2L)^n\quad\text{ and } \quad Q_L=[L,2L)\times [0,L)^{n-1}\,.
\]
Then $Q_L$ is a dyadic subcube of $R_L$ and $|R_L|=2^n|Q_L|$.

 We first observe that the largest dyadic subcube of $R_L$ that is contained in $B(0,1)$ has {side length} $\ge 1/(4\sqrt{n})$.
 Hence
 \begin{equation}\label{eq:M_RL-lower}
|\mathcal M(R_L)|\ge c(n)>0.
\end{equation}

 On the other hand,  every $f_\gamma(E)$-free dyadic subcube of $Q_L$ lies in  some  annulus $ B(0, (m+1)^{\gamma})\setminus \iol{B}(0, m^{\gamma}) $,
 where $m^{\gamma}\ge L$. So  {its side length} cannot exceed
 \[
 (m+1)^\gamma -m^\gamma \le \gamma m^{\gamma-1}\le\gamma L^{(\gamma-1)/\gamma}= \gamma 2^{k(\gamma-1)}
 ;
 \]
 here the first estimate follows from the mean value theorem (recall $\gamma <1$). Consequently,
\[
|\mathcal M(Q_L)|\le \gamma^n\,2^{-kn(1-\gamma)}\to 0
\]
as $k\to \infty$.

This together with \eqref{eq:M_RL-lower} reveals that the doubling property of Lemma~\ref{lem:M-doubling}
cannot hold for $f_\gamma(E)$. Thus $f_\gamma(E)$ is not weakly porous.
\end{proof}

\begin{proof}[Proof for Theorem~\ref{thm:sharpness}]
Since $E$ is weakly porous (Lemma~\ref{lem:E-weakly-porous}), we have \[\delta_E^{-\alpha}\in A_1\subset A_\infty\] for some $\alpha>0$. Hence $\log\delta_E\in\mathrm{BMO}$ and so $E$ is median porous. Set \[\gamma=1/K\in(0,1)\,.\] The radial stretching $f_\gamma$ is $K$-quasiconformal, so $f_\gamma(E)$ is median porous by Theorem~\ref{thm:main}.
However, the set $f_\gamma(E)$ is not weakly porous (Lemma~\ref{lem:image-not-weakly-porous}). This completes the proof.
\end{proof}

\begin{remark}
In dimension one, there are known  examples of sets that are median porous but not weakly porous~\cite{MR4773553,PU25}.
Theorem~\ref{thm:sharpness} extends the one-dimensional example
 \[\{\pm m^\gamma:m=1,2,3,\dots\}\subset\R\] from~\cite[Theorem~9.1]{PU25} to higher dimensions.
\end{remark}

\bibliographystyle{abbrv}

\begin{thebibliography}{10}

\bibitem{MR4065111}
T.~Adamowicz, K.~F\"assler, and B.~Warhurst.
\newblock A {K}oebe distortion theorem for quasiconformal mappings in the
  {H}eisenberg group.
\newblock {\em Ann. Mat. Pura Appl. (4)}, 199(1):147--186, 2020.

\bibitem{MR4773553}
T.~C. Anderson, J.~Lehrb\"ack, C.~Mudarra, and A.~V. V\"ah\"akangas.
\newblock Weakly porous sets and {M}uckenhoupt {$A_p$} distance functions.
\newblock {\em J. Funct. Anal.}, 287(8):Paper No. 110558, 34, 2024.

\bibitem{MR777305}
K.~Astala and F.~W. Gehring.
\newblock Quasiconformal analogues of theorems of {K}oebe and
  {H}ardy-{L}ittlewood.
\newblock {\em Michigan Math. J.}, 32(1):99--107, 1985.

\bibitem{MR1622690}
V.~I. Burenkov.
\newblock {\em Sobolev spaces on domains}, volume 137 of {\em Teubner-Texte zur
  Mathematik [Teubner Texts in Mathematics]}.
\newblock B. G. Teubner Verlagsgesellschaft mbH, Stuttgart, 1998.

\bibitem{MR807149}
J.~Garc\'ia-Cuerva and J.~L. Rubio~de Francia.
\newblock {\em Weighted norm inequalities and related topics}, volume 116 of
  {\em North-Holland Mathematics Studies}.
\newblock North-Holland Publishing Co., Amsterdam, 1985.
\newblock Notas de Matem\'atica, 104. [Mathematical Notes].

\bibitem{MR4966412}
I.~G\'omez~Vargas.
\newblock New characterizations of {M}uckenhoupt {$A_p$} distance weights for
  {$p>1$}.
\newblock {\em J. Math. Anal. Appl.}, 556(1):Paper No. 130091, 27, 2026.

\bibitem{MR1859913}
T.~Iwaniec and G.~Martin.
\newblock {\em Geometric function theory and non-linear analysis}.
\newblock Oxford Mathematical Monographs. The Clarendon Press, Oxford
  University Press, New York, 2001.

\bibitem{MR131498}
F.~John and L.~Nirenberg.
\newblock On functions of bounded mean oscillation.
\newblock {\em Comm. Pure Appl. Math.}, 14:415--426, 1961.

\bibitem{MR3453992}
R.~Korte, N.~Marola, and O.~Saari.
\newblock Homeomorphisms of the {H}eisenberg group preserving {BMO}.
\newblock {\em Arch. Math. (Basel)}, 106(2):175--182, 2016.

\bibitem{PU25}
M.~Pasquariello and I.~Uriarte-Tuero.
\newblock Medians, oscillations, and distance functions.
\newblock Preprint, arXiv:2507.21020v2, 2025.

\bibitem{MR361067}
H.~M. Reimann.
\newblock Functions of bounded mean oscillation and quasiconformal mappings.
\newblock {\em Comment. Math. Helv.}, 49:260--276, 1974.

\bibitem{MR816512}
V.~V. Shan{\cprime}kov.
\newblock The averaging operator with variable radius, and the inverse trace
  theorem.
\newblock {\em Sibirsk. Mat. Zh.}, 26(6):141--152, 191, 1985.

\bibitem{MR1232192}
E.~M. Stein.
\newblock {\em Harmonic analysis: real-variable methods, orthogonality, and
  oscillatory integrals}, volume~43 of {\em Princeton Mathematical Series}.
\newblock Princeton University Press, Princeton, NJ, 1993.
\newblock With the assistance of Timothy S. Murphy, Monographs in Harmonic
  Analysis, III.

\bibitem{MR442226}
A.~Uchiyama.
\newblock Weight functions of the class {$(A\sb{\infty })$} and quasi-conformal
  mappings.
\newblock {\em Proc. Japan Acad.}, 51:811--814, 1975.

\bibitem{MR454009}
J.~V\"ais\"al\"a.
\newblock {\em Lectures on {$n$}-dimensional quasiconformal mappings}, volume
  229 of {\em Lecture Notes in Mathematics}.
\newblock Springer-Verlag, Berlin-New York, 1971.

\bibitem{MR869219}
J.~V\"ais\"al\"a.
\newblock Porous sets and quasisymmetric maps.
\newblock {\em Trans. Amer. Math. Soc.}, 299(2):525--533, 1987.

\bibitem{Vasin2003}
A.~V. Vasin.
\newblock The limit set of a {F}uchsian group and the {D}yn\cprime kin lemma.
\newblock {\em Zap. Nauchn. Sem. S.-Peterburg. Otdel. Mat. Inst. Steklov.
  (POMI)}, 303(Issled. po Line\u{\i}n. Oper. i Teor. Funkts. 31):89--101, 322,
  2003.
\newblock English transl. in J. Math. Sci. (N.Y.) {\bf 129} (2005), no. 4,
  3977--3984.

\end{thebibliography}

\end{document}